\documentclass{amsart}
\usepackage{url}
\usepackage{amsmath}
\usepackage{amssymb}

\newtheorem{thm}{Theorem}[section]
\newtheorem{prop}[thm]{Proposition}
\newtheorem{lem}[thm]{Lemma}
\newtheorem{cor}[thm]{Corollary}

\theoremstyle{definition}
\newtheorem{definition}[thm]{Definition}

\theoremstyle{remark}

\numberwithin{equation}{section}

\newcommand{\ip}[2]{\left\langle {#1} , {#2} \right\rangle} %inner product
\newcommand{\norm}[1]{\left\lVert #1 \right\rVert} %standard norm
\newcommand{\R}{\mathbb{R}} %reals
\newcommand{\N}{\mathbb{N}} %naturals
\newcommand{\C}{\mathbb{C}} %complexs

\begin{document}

%%
%% The title of the paper goes here.  Edit to your title.
%%

\title{A Szeg\H{o} Limit Theorem for Radially-Compressed Toeplitz Operators}

%%
%% Now edit the following to give your name and address:
%% 

\author{Trevor Camper}
\address{School of Mathematical and Statistical Sciences, Clemson University,
Clemson, SC 29630}
\email{trcampe@g.clemson.edu}

%%
%% If there is another author uncomment and edit the following.
%%

% \author{Mishko Mitkovski}
% \address{School of Mathematical and Statistical Sciences, Clemson University,
% Clemson, SC 29630}

%%
%% If there are three of more authors they are added in the obvious
%% way. 
%%

%%%
%%% The following is for the abstract.  The abstract is optional and
%%% if not used just delete, or comment out, the following.
%%%

\begin{abstract}
    We obtain Szeg\H{o}-type Limit Theorems in the setting of Reproducing Kernel Hilbert Spaces on discs in $\C$. From this, we derive a formula for the density of the eigenvalues of compressions of Toeplitz operators. Examples for the Bergman and Segal-Bargmann-Fock space are also presented.   
\end{abstract}

%%
%%  LaTeX will not make the title for the paper unless told to do so.
%%  This is done by uncommenting the following.
%%

\maketitle

%%
%% LaTeX can automatically make a table of contents.  This is done by
%% uncommenting the following:
%%

%\tableofcontents

%%
%%  To enter text is easy.  Just type it.  A blank line starts a new
%%  paragraph. 
%%

%%
%%  To put mathematics in a line it is put between dollor signs.  That
%%  is $(x+y)^2=x^2+2xy+y^2$
%%

%%
%%% Displayed mathematics is put between double dollar signs.  
%%

%%
%% Its proof is set off by
%% 
%%
%% A new section is started as follows:
%%
%%%%%%%%%%%%%%%%%%%%%%%%%%%%%%%%%%%%%%%%%%%%%%%%%%%%%%%%%%%%%%%%%%%%%%
\section{Introduction}
%%%%%%%%%%%%%%%%%%%%%%%%%%%%%%%%%%%%%%%%%%%%%%%%%%%%%%%%%%%%%%%%%%%%%%
Let $c:=\{c_{n}\}_{n\geq 0}$ be a strictly positive sequence, i.e. $c_{n}>0$ for each $n\in\N\cup\{0\}$, such that the sequence $\{c_{n+1}/c_{n}\}$ is monotonic. Let 
\begin{align*}
    R=\lim_{n\to\infty}\frac{c_{n+1}}{c_{n}},
\end{align*}
where $R=\infty$ if the corresponding limit diverges. We assume that  $c:=\{c_{n}\}_{n\geq 0}$ is the moment sequence of an absolutely continuous, finite measure $d\mu(r)=\mu(r)dr$ with $\text{supp}(\mu)=[0,R]$, i.e.
\begin{align*}
    c_{n}=\int_{0}^{R}r^{n}\mu(r)dr.
\end{align*}

Denote by $\mathcal{H}_{c}$ the space of analytic functions with expansion 
\begin{align*}
    f(z)=\sum_{n=0}^{\infty}a_{n}z^{n}
\end{align*}
for $|z|<R$ such that 
\begin{align*}
    \norm{f}^{2}_{c}:=\sum_{n=0}^{\infty}|a_{n}|^{2}c_{2n+1}<\infty.
\end{align*}
$\mathcal{H}_{c}$ is a Reproducing Kernel Hilbert space (RKHS) with inner product inherited from $\norm{\cdot}_{c}$. Note that for the measure $d\nu(z)=\mu(|z|)\frac{dz}{2\pi}$ we see 
\begin{align*}
    \ip{\frac{z^{n}}{\sqrt{c_{2n+1}}}}{\frac{z^{\ell}}{\sqrt{c_{2\ell+1}}}}_{L^{2}(\nu)}=\frac{1}{\sqrt{c_{2k+1}c_{2\ell+1}}}\int_{0}^{R}r^{k+\ell+1}d\mu(r)\int_{0}^{2\pi}e^{i(k-\ell)\theta}\frac{d\theta}{2\pi}=\delta_{\ell,k}.
\end{align*}
Thus for $f\in\mathcal{H}_{c}$ we see 
\begin{align*}
    \ip{f}{\frac{z^{k}}{\sqrt{c_{2k+1}}}}_{L^{2}(\nu)}=a_{k}\frac{1}{\sqrt{c_{2k+1}}}\int_{0}^{R}r^{2k+1}d\mu(r)=a_{k}\sqrt{c_{2k+1}}.
\end{align*}
Hence, 
\begin{align*}
    \norm{f}^{2}_{c}=\sum_{n=0}^{\infty}|a_{k}|^{2}c_{2k+1}=\sum_{n=0}^{\infty}\left|\ip{f}{\frac{z^{n}}{\sqrt{c_{2n+1}}}}_{L^{2}(\nu)}\right|^{2}=\norm{f}^{2}_{L^{2}(\mathbb{D}(0,R),\nu)}.
\end{align*}
From this, we can see that $\mathcal{H}_{c}$ can be viewed as a closed subspace of $L^{2}(\mathbb{D}(0,R),\nu)$. The most interesting such spaces are when $\mu(r)=2$ (i.e., $c_{2n+1}=(n+1)^{-1}$ and $\mathcal{H}_{c}$ is the Bergman space $A^{2}(\mathbb{D})$), and when $\mu(r)=2e^{-r^{2}}$ (i.e., $c_{2n+1}=n!$ and $\mathcal{H}_{c}$ is the Segal-Bargmann-Fock Space $\mathcal{F}^{2}(\C)$). Our main object of study will be Toeplitz operators $T_{\sigma}$ with $\sigma\in L^{\infty}(\mathbb{D}(0,R))$ acting on $\mathcal{H}_{c}$. 
\begin{definition}
    Let $\sigma\in L^{\infty}(\mathbb{D}(0,R),\nu)$, and let $P:L^{2}(\mathbb{D}(0,R),\nu)\to\mathcal{H}_{c}$ be the projection onto $\mathcal{H}_{c}$. A Toeplitz operator $T_{\sigma}:\mathcal{H}_{c}\to\mathcal{H}_{c}$ is the operator $PM_{\sigma}$, where $M_{\sigma}$ is the multiplication operator on $L^{2}(\mathbb{D}(0,R),\nu)$ associated to $\sigma$.  
\end{definition}
The spectral asymptotics of operators is a well studied field, see \cite{guille79} for a history of the subject and \cite{DMFN02,FN01} for the case of Toeplitz operators acting on spaces similar to $\mathcal{H}_{c}$. In the case of the Hardy space $\mathcal{H}^{2}(\mathbb{D})$, one of the most classical result concerning the spectral asymptotics of Toeplitz operators is the Szeg\H{o} Limit Theorem \cite{SZ15}: if $\psi:\mathbb{R}\to\mathbb{R}$ is continuous, and $\sigma:\mathbb{T}\to\R$ is continuous, then 
\begin{align*}
    \lim_{N\to\infty}\frac{1}{N}\text{tr}\left(\psi\left(P_{N}T_{\sigma}P_{N}\right)\right)=\frac{1}{2\pi}\int_{0}^{2\pi}\psi\left(\sigma(x)\right)dx, \tag{1}
\end{align*}
where $P_{N}$ denotes the projection onto $\text{span}\{e^{ik\theta}:k=0,\pm 1,\cdots,\pm N\}$. In this note, we will be interested in limits of a similar form. In the context of the Bergman and Segal-Bargmann-Fock spaces, the radial Toeplitz operators are those Toeplitz operators which are diagonalized by the monomials $z^{n}$. In keeping with the classical case, we will be concerned with compressions onto these eigenfunctions which motivates the following definition. 
\begin{definition}
    Let $\mathcal{K}_{N}=\text{span}\{\frac{z^{n}}{\sqrt{c_{2n+1}}}:n=0,1,\cdots,N\}$, let $P_{N}$ denote the orthogonal projection onto $\mathcal{K}_{N}$, and let $T_{\sigma}$ denote the Toeplitz operator with symbol $\sigma\in L^{\infty}(\mathbb{D})$ acting on $\mathcal{H}_{c}$. We call the operator $P_{N}T_{\sigma}P_{N}$ a radially-compressed Toeplitz operator.  
\end{definition}
As compared to the classical Szeg\H{o} Limit Theorem, we will obtain results of the form: for $\psi:\mathbb{R}\to\C$ and $\sigma:\mathbb{D}(0,R)\to\R$ continuous,  
\begin{align*}
    \lim_{N\to\infty}\frac{1}{N}\text{tr}\left(\psi\left(P_{N}T_{\sigma}P_{N}\right)\right)=\lim_{r\to R}\frac{1}{2\pi}\int_{0}^{2\pi}\psi\left(\sigma(re^{i\theta})\right)d\theta.
\end{align*}
 Essentially, this statement says that the asymptotic spectral averages of radially-compressed Toeplitz operators is determined by the boundary behavior of the symbol. From this result, as is standard in the literature, we obtain the following Weyl law: for $\lambda>0$, 
\begin{align*}
    \lim_{N\to\infty}\frac{\#\{0\leq j\leq N:\lambda_{j}^{N}(\sigma)>\lambda\}}{N+1}=\frac{|\{\Tilde{\sigma}>\lambda\}|}{2\pi},
\end{align*}
where $\{\lambda_{j}^{N}(\sigma)\}_{j}$ are the eigenvalues of the compression $P_{N}T_{\sigma}P_{N}$.

The structure of this note is as follows. In Section 2, we will prove some preliminary facts that will be used throughout. In Section 3, we will present the proof strategy for our Szeg\H{o} limit and obtain the associated spectral density for $\mathcal{H}_{c}$. In Section 4, we present the corresponding results for the Bergman and Segal-Bargmann-Fock spaces. 
\section{Preliminaries}
The strategy for proving our results relies on the following proposition proved below. Essentially, the vectors $z^{n}/\sqrt{c_{2n+1}}$ give rise to probability measures $\mu_{n}$ which experience mass escaping at the boundary of $\mathbb{D}(0,R)$ and hence sample $\sigma$ along this boundary. This motivates the following definition.
\begin{definition}
    Let $\mu_{n}$ be the family of probability measures on $[0,R)$ satisfying $d\mu_{n}(r)=\frac{1}{c_{2n+1}}r^{2n+1}d\mu(r)$. 
\end{definition}
Following this, we prove an approximation result which allows us to compare the trace of our radially-compressed Toeplitz operator with the trace of a compressed Toeplitz operator whose symbol is independent of the magnitude of the argument. Finally, identifying this Toeplitz operator with a matrix and invoking a general Szeg\H{o} limit result of \cite{bourget2017first} gives the result. First, we make rigorous this notion of mass escaping at the boundary with the following lemma.
\begin{prop}
    Let $\mu_{n}$ be defined as above. Then, 
    \begin{itemize}
        \item[1.] For any $\Tilde{R}<R$, we have $\mu_{n}\left([0,\Tilde{R})\right)\to 0$ as $n\to\infty$. 
        \item[2.] If $R<\infty$, then for any $m\in\N$ 
        \begin{align*}
            \frac{c_{2l+m+1}}{\sqrt{c_{2l+2m+1}}\sqrt{c_{2l+1}}}\to 1,
        \end{align*}
        as $l\to\infty$. 
    \end{itemize}
\end{prop}
\begin{proof}
    For the first point, choose $\Tilde{R}<R_{1}<R$. Then, 
    \begin{align*}
        0\leq \mu_{n}\left([0,\Tilde{R})\right)=\frac{\int_{0}^{\Tilde{R}}r^{2n+1}d\mu(r)}{\int_{0}^{R}r^{2n+1}d\mu(r)}\leq\frac{\mu([0,\Tilde{R}))\left(\Tilde{R}\right)^{2n+1}}{\int_{R_{1}}^{R}r^{2n+1}d\mu(r)}\lesssim\left(\frac{\Tilde{R}}{R_{1}}\right)^{2n+1}\to0
    \end{align*}
    as $n\to\infty$. For the second point, note that 
    \begin{align*}
        \frac{c_{2l+m+1}}{\sqrt{c_{2l+2m+1}}\sqrt{c_{2l+1}}}=\frac{\frac{c_{2l+m+1}}{c_{2l+1}R^{m}}}{\sqrt{\frac{c_{2l+2m+1}}{c_{2l+1}R^{2m}}}}=\frac{\frac{c_{2l+2}}{c_{2l+1}R}\cdots\frac{c_{2l+m+1}}{c_{2l+m}R}}{\sqrt{\frac{c_{2l+2}}{c_{2l+1}R}\cdots\frac{c_{2l+2m+1}}{c_{2l+2m}R}}}\to 1,
    \end{align*}
    by the radius of convergence assumption. 
\end{proof}
When $R=\infty$, we make the additional assumption that Part 2 of Proposition 2.2 is true. It is not obvious to the author whether or not this is true in general. However, the result is true in the case of the Segal-Bargmann-Fock space. Continuing on, the following lemma captures the mass escaping at the boundary property.  
\begin{lem}
    The family of probability measures $d\mu_{n}$ on $[0,R)$ has a weak limit $\delta_{R}$ in the following sense: if $\sigma\in L^{\infty}([0,R))$ such that $\sigma(r)\to c<\infty$ as $r\to R^{-}$, then 
    \begin{align*}
        \int_{0}^{R}\sigma(r)d\mu_{n}(r)\to c,
    \end{align*}
    as $n\to\infty$. 
\end{lem}
\begin{proof}
    Let $\sigma\in L^{\infty}([0,R))$ such that $\sigma(r)\to c<\infty$ as $r\to R^{-}$. Let $\varepsilon>0$. Then there is an $\Tilde{R}<R$ such that $|\sigma(r)-c|<\varepsilon/2$ when $r\in [\Tilde{R},R)$. Furthermore, by Proposition 2.2, there is an $N$ such that for $n\geq N$
    \begin{align*}
        \mu_{n}\left([0,\Tilde{R})\right)<\frac{\varepsilon}{2\max\{c,\norm{\sigma}_{\infty}\}}.
    \end{align*}
    Then, 
    \begin{align*}
        \left|\int_{0}^{R}\sigma(r)d\mu_{n}(r)-c\right|\leq&\left|\int_{0}^{\Tilde{R}}\left(\sigma(r)-c\right)d\mu_{n}(r)\right|+\left|\int_{\Tilde{R}}^{R}\left(\sigma(r)-c\right)d\mu_{n}(r)\right|\\
        \leq&2\text{max}\{c,\norm{\sigma}_{\infty}\}\mu_{n}\left([0,\Tilde{R})\right)+\varepsilon/2\int_{\Tilde{R}}^{R}d\mu_{n}(r)\\
        <&\varepsilon,
    \end{align*}
    and the result is established. 
\end{proof}
The Szeg\H{o} limits presented in the introduction can be interpreted as spectral averages in some sense, which we will take advantage of going forward. In particular, we will need a weak convergence for an average of measures which our next lemma provides. 
\begin{lem}
    Suppose $\sigma\in L^{\infty}([0,R))$ is such that $\sigma(r)\to c$ as $r\to R^{-}$. Let $\mu_{n}$ be as above. Then, 
    \begin{align*}
        \frac{1}{N+1}\sum_{n=0}^{N}\int_{0}^{R}\sigma(r)d\mu_{n}(r)\to c,
    \end{align*}
    as $N\to\infty$. 
\end{lem}
\begin{proof}
    This is a simple consequence of Lemma 3.1 and Ces\`aro convergence of a sequence. 
\end{proof}
We now make the following simple observation. Note that 
\begin{align*}
    \ip{T_{\sigma}\frac{z^{n}}{\sqrt{c_{2n+1}}}}{\frac{z^{n}}{\sqrt{c_{2n+1}}}}=\int_{\mathbb{D}(0,R)}\sigma(z)\frac{1}{c_{2n+1}}|z|^{2n}d\nu(z)=\int_{0}^{R}\hat{\sigma}(r)d\mu_{n}(r),
\end{align*}
where 
\begin{align*}
    \hat{\sigma}(r)=\int_{0}^{2\pi}\sigma(re^{i\theta})\frac{d\theta}{2\pi}.
\end{align*}
Hence, if $\lim_{r\to R^{-}}\hat{\sigma}(r)$ exists, we may apply the above results to obtain the required Szeg\H{o} limit. This motivates the following definition.
\begin{definition}
    Let $\sigma\in L^{\infty}(\mathbb{D}(0,R))$ be real-valued. We say $\sigma$ has a radial limit if $\lim_{r\to R^{-}}\sigma(re^{i\theta})$ is well defined and we write $\Tilde{\sigma}(\theta):=\lim_{r\to R^{-}}\sigma(re^{i\theta})$. 
\end{definition}
It is well known that such limits exist pointwise a.e. in the holomorphic and continuous cases. We now begin to construct our Szeg\H{o} Limit Theorem. 
\section{A Szeg\H{o}-type Limit Theorem}
\begin{prop}[Averaging Theorem]
    Suppose $\sigma\in L^{\infty}(\mathbb{D}(0,R))$ has a radial limit $\Tilde{\sigma}$. Then,  
    \begin{align*}
        \lim_{N\to\infty}\frac{1}{N+1}\text{tr}\left(P_{N}T_{\sigma}P_{N}\right)=\int_{0}^{2\pi}\Tilde{\sigma}(\theta)\frac{d\theta}{2\pi}.
    \end{align*} 
\end{prop}
\begin{proof}
    First, note that 
    \begin{align*}
        \frac{1}{N+1}\text{tr}\left(P_{N}T_{\sigma}P_{N}\right)=&\frac{1}{N+1}\sum_{n=0}^{N}\ip{T_{\sigma}\frac{z^{n}}{\sqrt{c_{2n+1}}}}{\frac{z^{n}}{\sqrt{c_{2n+1}}}}\\
        =&\frac{1}{N+1}\sum_{n=0}^{N}\int_{\mathbb{D}(0,R)}\sigma(z)\frac{1}{c_{2n+1}}|z|^{2n}\mu(|z|)\frac{dz}{\pi}\\
        =&\frac{1}{N+1}\sum_{n=0}^{N}\int_{0}^{R}\hat{\sigma}(r)d\mu_{n}(r),
    \end{align*}
    where 
    \begin{align*}
        \hat{\sigma}(r)=\int_{0}^{2\pi}\sigma(re^{i\theta})\frac{d\theta}{2\pi}.
    \end{align*}
    Applying Lemmas 2.3 and 2.4 give the result. 
\end{proof}
Before proceeding to our Szeg\H{o} limit result and the spectral density, we require two approximation results. The first, which is essentially a result of Janssen and Zelditch \cite{janssen1983szegHo}, gives an approximation of the difference of the traces of monomials of the radially-compressed Toeplitz operators.
\begin{lem}
    Let $\sigma,\eta\in L^{\infty}(\mathbb{D}(0,R))$, $\sigma,\eta$ real valued and let $P_{N}$ be as above. Then, for each $k\in\N$
    \begin{align*}
        \left|\text{tr}\left(\left(P_{N}T_{\sigma}P_{N}\right)^{k}\right)-\text{tr}\left(\left(P_{N}T_{\eta}P_{N}\right)^{k}\right)\right|\lesssim_{k}&\sum_{n=0}^{N}\int_{\mathbb{D}(0,R)}|\sigma(z)-\eta(z)|\frac{1}{c_{2n+1}}|z|^{2n}d\nu(z).
    \end{align*}
\end{lem}
\begin{proof}
    The first part of this proof comes from Janssen and Zelditch \cite{janssen1983szegHo}. Note that for any two bounded operators $A$ and $B$, $A^{k}-B^{k}=\left((A-B)+B\right)^{k}-B^{k}=\text{sum of terms}$, each of which has at least one factor of $A-B$. Hence, using the inequalities $\norm{AB}_{1}\leq\norm{A}\norm{B}_{1}$, we obtain
    \begin{align*}
        \left|\text{tr}\left(\left(P_{N}T_{\sigma}P_{N}\right)^{k}\right)-\text{tr}\left(\left(P_{N}T_{\eta}P_{N}\right)^{k}\right)\right|\leq&\norm{\left(P_{N}T_{\sigma}P_{N}\right)^{k}-\left(P_{N}T_{\eta}P_{N}\right)^{k}}_{1}\\
        \leq&C\norm{P_{N}\left(T_{\sigma}-T_{\eta}\right)P_{N}}_{1}\\
        =&C\norm{P_{N}T_{\sigma-\eta}P_{N}}_{1},
    \end{align*}
    where $C$ is a constant independent of $N$. Now, noting that $(\sigma-\eta)=(\sigma-\eta)_{+}-(\sigma-\eta)_{-}$, and applying the triangle inequality for the trace norm twice, we obtain
    \begin{align*}
        C\norm{P_{N}T_{\sigma-\eta}P_{N}}_{1}\lesssim&\norm{P_{N}T_{(\sigma-\eta)_{+}}P_{N}}+\norm{P_{N}T_{(\sigma-\eta)_{-}}P_{N}}\\
        =&\text{tr}\left(P_{N}T_{\left(\sigma-\eta\right)_{+}}P_{N}\right)+\text{tr}\left(P_{N}T_{\left(\sigma-\eta\right)_{-}}P_{N}\right)\\
        =&\text{tr}\left(P_{N}T_{|\sigma-\eta|}P_{N}\right)\\
        =&\sum_{n=0}^{N}\int_{\mathbb{D}(0,R)}|\sigma(z)-\eta(z)|\frac{1}{c_{2n+1}}|z|^{2n}d\nu(z).
    \end{align*}
\end{proof}
Our second approximation result is standard, and will be used for a Stone-Weierstrauss argument in our main result. The result is standard, and we include it without proof.
\begin{lem}
    Let $\mathcal{H}$ be a finite-dimensional Hilbert space. Suppose $\psi:\mathbb{R}\to\mathbb{C}$ is continuous, and $A:\mathcal{H}\to\mathcal{H}$ is bounded and self-adjoint. Then, 
    \begin{align*}
        \left|\text{tr}\left(\psi\left(A\right)\right)\right|\leq&\text{dim}(\mathcal{H})\sup\{|\psi(r)|:r\in\text{spec}(A)\},
    \end{align*}
    where $\text{spec}(A)$ denotes the spectrum of $A$. 
\end{lem}
Finally, combining the above lemmas we arrive at the main result of this note.
\begin{thm}
    Suppose $\sigma\in L^{\infty}(\mathbb{D}(0,R))$ is real-valued and has a radial limit. Then, for any continuous $\psi:[\inf\sigma,\sup\sigma]\to\mathbb{C}$ 
    \begin{align*}
        \lim_{N\to\infty}\frac{1}{N+1}\text{tr}\left(\psi\left(P_{N}T_{\sigma}P_{N}\right)\right)=\int_{0}^{2\pi}\psi\left(\Tilde{\sigma}(\theta)\right)\frac{d\theta}{2\pi}.
    \end{align*}
\end{thm}
\begin{proof}
    We will first prove the result for monomials, and then employ a Stone-Weierstrauss argument to transition to continuous functions, via Lemma $3.3$. Let $k\in\N$, as the result is trivially true for $k=0$. Let $\Tilde{\sigma}$ denote the radial limit of $\sigma$ as a function on $\mathbb{D}(0,R)$, i.e. $\Tilde{\sigma}(z)=\Tilde{\sigma}({arg}(z))$. Then, applying Lemma $3.2$, we have 
    \begin{align*}
        &\left|\frac{1}{N+1}\text{tr}\left(\left(P_{N}T_{\sigma}P_{N}\right)^{k}\right)-\frac{1}{N+1}\text{tr}\left(\left(P_{N}T_{\Tilde{\sigma}}P_{N}\right)^{k}\right)\right|\\
        \lesssim&\frac{1}{N+1}\sum_{n=0}^{N}\int_{\mathbb{D}(0,R)}|\sigma(z)-\Tilde{\sigma}(z)|\frac{1}{c_{2n+1}}|z|^{2n}d\nu(z)\\
        =&\frac{1}{N+1}\sum_{n=0}^{N}\int_{0}^{R}\left(\int_{0}^{2\pi}|\sigma(\sqrt{r}e^{i\theta})-\Tilde{\sigma}(e^{i\theta})|\frac{d\theta}{2\pi}\right)d\mu_{n}(r).
    \end{align*}
    Thus, by applying the given assumptions along with Lemma $2.4$, we have that this term goes to zero as $N\to\infty$. Hence, 
    \begin{align*}
       \frac{1}{N+1}\text{tr}\left(\left(P_{N}T_{\sigma}P_{N}\right)^{k}\right)\sim \frac{1}{N+1}\text{tr}\left(\left(P_{N}T_{\Tilde{\sigma}}P_{N}\right)^{k}\right).
    \end{align*}
   We now consider the right-most term. The operator $P_{N}T_{\Tilde{\sigma}}P_{N}$ can be identified with a matrix whose $(k,\ell)$ entry is 
   \begin{align*}
       \hat{\sigma}(k-l)\frac{c_{k+\ell+1}}{\sqrt{c_{2k+1}}\sqrt{c_{2\ell+1}}}.
   \end{align*}
    However, the result holds in this case by appealing to \cite{bourget2017first} (see the Theorem 3.5 in the case of density-1 sequences) and Lemma $3.2$ by either Proposition $2.2$ when $R<\infty$ or the assumption when $R=\infty$. Thus, the result is established for monomials. The result is further established for any complex polynomial due to the additivity of the trace. Finally, the result in full generality follows by appealing to Lemma $3.3$ and applying the Stone-Weierstrauss Theorem on the set $[\inf\sigma,\sup\sigma]$.
\end{proof}
Using an approximation argument of \cite{FN01} we can obtain the following general spectral density result. 
\begin{cor}
    Let $\inf\sigma<\alpha<\beta<\sup\sigma$ with either $\alpha>0$ or $\beta<0$, and suppose $\sigma\in L^{\infty}(\mathbb{D})$ is real-valued such that $|\{\Tilde{\sigma}=\alpha\}|=|\{\Tilde{\sigma}=\beta\}|=0$. Then, 
    \begin{align*}
        \lim_{N\to\infty}\frac{\#\{0\leq j\leq N:\alpha<\lambda_{j}^{N}(\sigma)<\beta\}}{N+1}=\frac{\left|\{\alpha<\Tilde{\sigma}(\theta)<\beta\}\right|}{2\pi}.
    \end{align*}
\end{cor}
\begin{proof}
    Without loss of generality, we assume that $\norm{\sigma}_{\infty}=1$. Let $\chi=\chi_{\alpha,\beta}$ be the indicator function on $(\alpha,\beta)$. Let $\varepsilon>0$. There are continuous functions $\psi_{\varepsilon,1},\psi_{\varepsilon,2}$ such that $\psi_{\varepsilon,1}(x)\leq\chi(x)\leq\psi_{\varepsilon,2}(x)\leq \frac{1}{\max{|\alpha|,|\beta|}}$ and $\psi_{\varepsilon,1}$ and $\psi_{\varepsilon,2}$ coincide with $\chi$ outside of the intervals $[\alpha-\varepsilon/2,\alpha+\varepsilon/2]$ and $[\beta-\varepsilon/2,\beta+\varepsilon/2]$. By Theorem 3.7, we have 
    \begin{align*}
        \frac{1}{N+1}\text{tr}\left(\left(\psi_{\varepsilon,2}-\psi_{\varepsilon,1}\right)\left(P_{N}T_{\sigma}P_{N}\right)\right)\to\int_{0}^{2\pi}\left(\psi_{\varepsilon,2}-\psi_{\varepsilon,1}\right)\left(\Tilde{\sigma}(\theta)\right)\frac{d\theta}{2\pi}.
    \end{align*}
    Let 
    \begin{align*}
        c(\gamma,\varepsilon)=\frac{|\{\gamma-\varepsilon/2\leq\Tilde{\sigma}(\theta)\leq\gamma+\varepsilon/2\}|}{2\pi}.
    \end{align*}
    Note that $c(\gamma,\varepsilon)\to 0$ as $\varepsilon\to0$ by the hypothesis, and  
    \begin{align*}
        \int_{0}^{2\pi}\left(\psi_{\varepsilon,2}-\psi_{\varepsilon,1}\right)\left(\Tilde{\sigma}(\theta)\right)\frac{d\theta}{2\pi}\leq\frac{c(\alpha,\varepsilon)+c(\beta,\varepsilon)}{2\pi}\frac{1}{\max{\{|\alpha|,|\beta|\}}}.
    \end{align*}
    Thus, letting $\varepsilon\to 0$ we see that the limit in the statement of this theorem exists, and must equal the desired quantity. 
\end{proof}
% Finally, we have the following distributional result concerning the zero distribution of holomorphic functions. It is well known that such functions have radial limits, and hence we may apply the above results. In particular, we obtain the following.
% \begin{cor}
%     Let $F:\mathbb{D}(0,1)\to\C$ be holomorphic with roots $\{a_{k}\}_{k=1}^{n}$ such that $|F(0)|\neq 0$. Then, 
%     \begin{align*}
%         \lim_{N\to\infty}\left(\text{det}\left(P_{N}T_{|F|}P_{N}\right)\right)^{1/N}=\frac{|F(0)|}{\Pi_{k=1}^{n}|a_{k}|}.
%     \end{align*}
% \end{cor}
% \begin{proof}
%     First, note that by Theorem 3.4 and Jensen's Formula \cite{ahlfors1979complex} in the case of $\psi(x)=\log(x)$ we have 
%     \begin{align*}
%         \lim_{N\to\infty}\frac{1}{N}\text{tr}\left(\log\left(P_{N}T_{|F|}P_{N}\right)\right)=\int_{0}^{2\pi}\log|F(e^{i\theta})|\frac{d\theta}{2\pi}=&\sum_{k=1}^{n}\log\left(\frac{1}{|a_{k}|}\right)+\log(|F(0)|)\\=&\log\left(\frac{|F(0)|}{\Pi_{k=1}^{n}|a_{k}|}\right).
%     \end{align*}
%     Exponentiating both sides and using continuity gives the result. 
% \end{proof}
\section{Applications}
We now cover two examples discussed in the introduction. 
\subsection{The Bergman Space $A^{2}(\mathbb{D})$}
The Bergman space $A^{2}(\mathbb{D})$ is the space of holomorphic functions which are square-integrable with respect to the normalized area measure on $\mathbb{D}$. It is known \cite{Z90} that $A^{2}(\mathbb{D})$ is a RKHS, and in this case we have $d\nu(z)=\frac{1}{\pi}dz$ and $c_{2n+1}=\frac{1}{n+1}$. Thus, we have the following corollary. 
\begin{cor}
    Suppose $\sigma\in L^{\infty}(\mathbb{D},\frac{1}{\pi}dz)$ is real-valued and has a radial limit. Then for any $\psi:[-\inf\sigma,\sup\sigma]\to\C$ continuous, 
    \begin{align*}
        \lim_{N\to\infty}\text{tr}\left(\psi\left(P_{N}T_{\sigma}P_{N}\right)\right)=\int_{0}^{2\pi}\psi\left(\Tilde{\sigma}(\theta)\right)\frac{d\theta}{2\pi}.
    \end{align*}
\end{cor}
The above corollary then leads to the following spectral density result. 
\begin{cor}
    Suppose $\sigma\in L^{\infty}(\mathbb{D},\frac{1}{\pi}dz)$ is real. Then, for any $\alpha,\beta\in(\inf\sigma,\sup\sigma)$ with $\alpha<\beta$ and either $\alpha>0$ or $\beta<0$ we have 
    \begin{align*}
        \lim_{N\to\infty}\frac{\#\{0\leq j\leq N:\alpha\leq \lambda_{j}^{N}(\sigma)\leq\beta\}}{N+1}=\frac{|\{\alpha\leq\Tilde{\sigma}(\theta)\leq\beta\}|}{2\pi}.
    \end{align*}
\end{cor}
We make the following observation. Let $\sigma(z)=|z|\cdot\text{arg}(z)$. Clearly, $\sigma$ satisfies the above assumptions. Thus, the doubly-indexed array $\{\lambda_{j}^{N}(\sigma)\}_{j\geq0,N\geq 0}$ must be equidistributed modulo $2\pi$. 
\subsection{The Segal-Bargmann-Fock Space $\mathcal{F}^{2}(\mathbb{C})$}
The Segal-Bargmann-Fock space $\mathcal{F}^{2}(\mathbb{C})$ is the space of entire functions which are square-integrable with respect to the Gaussian measure $\exp{\left(-|z|^{2}\right)}\frac{dz}{\pi}$. It is known \cite{Z12} that $\mathcal{F}^{2}(\mathbb{C})$ is a RKHS, and in this case we have $c_{2n+1}=n!$. In this case, we must check the condition on the moments discussed at the beginning of Section 3. However, this follows easily from the following property of the Gamma function (see 5.11.12 of \cite{NIST:DLMF}):
\begin{align*}
    \lim_{x\to\infty}\frac{\Gamma(x+\alpha)}{x^{\alpha}\Gamma(x)}=1
\end{align*}
for any $\alpha\in\R$. In fact, note that this condition is equivalent to 
\begin{align*}
    \frac{\Gamma(x+m/2+1)}{\Gamma(x+m+1)\Gamma(x+1)}\to 1
\end{align*}
as $x\to\infty$. However, 
\begin{align*}
    \frac{\Gamma(x+m/2+1)}{\Gamma(x+m+1)\Gamma(x+1)}=\frac{\frac{\Gamma(x+m/2+1)}{x^{m/2+1}\Gamma(x)}}{\sqrt{\frac{\Gamma(x+m+1)}{x^{m+1}\Gamma(x)}}\sqrt{\frac{\Gamma(x+1)}{x\Gamma(x)}}},
\end{align*}
and the result follows. Thus, we have the following corollary.
\begin{cor}
    Suppose $\sigma\in L^{\infty}(\mathbb{D},\frac{1}{\pi}e^{-|z|^{2}}dz)$ is real-valued and has a radial limit. Then for any continuous $\psi:[\inf\sigma,\sup\sigma]\to\C$, 
    \begin{align*}
        \lim_{N\to\infty}\text{tr}\left(\psi\left(P_{N}T_{\sigma}P_{N}\right)\right)=\int_{0}^{2\pi}\psi\left(\Tilde{\sigma}(\theta)\right)\frac{d\theta}{2\pi}.
    \end{align*} 
\end{cor}
Finally, from the above corollary, we obtain the following spectral density result. 
\begin{cor}
    Suppose $\sigma\in L^{\infty}(\mathbb{D},\frac{1}{\pi}e^{-|z|^{2}}dz)$ is real-valued and has a radial limit. Then, for any $\alpha,\beta\in(\inf\sigma,\sup\sigma)$ with $\alpha<\beta$ and either $\alpha>0$ or $\beta<0$ we have 
    \begin{align*}
        \lim_{N\to\infty}\frac{\#\{0\leq j\leq N:\alpha< \lambda_{j}^{N}(\sigma)<\beta\}}{N+1}=\frac{|\{\alpha<\Tilde{\sigma}(\theta)<\beta\}|}{2\pi}.
    \end{align*}
\end{cor}
\bibliographystyle{plain}
\bibliography{refs}

\begin{thebibliography}{1}

\bibitem{bourget2017first}
A~Bourget and T~McMillen.
\newblock A {F}irst {S}zeg{\H{o}}’s {L}imit {T}heorem for a {C}lass of non-{T}oeplitz {M}atrices.
\newblock {\em Constructive Approximation}, 45(1):47--63, 2017.

\bibitem{DMFN02}
Filippo De~Mari, Hans~G. Feichtinger, and Krzysztof Nowak.
\newblock Uniform {E}igenvalue {E}stimates for {T}ime-{F}requency {L}ocalization {O}perators.
\newblock {\em J. London Math. Soc.}, 65(3):720--732, 2002.

\bibitem{NIST:DLMF}
{\it NIST Digital Library of Mathematical Functions}.
\newblock \url{https://dlmf.nist.gov/}, Release 1.2.2 of 2024-09-15.
\newblock F.~W.~J. Olver, A.~B. {Olde Daalhuis}, D.~W. Lozier, B.~I. Schneider, R.~F. Boisvert, C.~W. Clark, B.~R. Miller, B.~V. Saunders, H.~S. Cohl, and M.~A. McClain, eds.

\bibitem{FN01}
Hans~G Feichtinger and K~Nowak.
\newblock A {S}zeg{\H{o}}-type {T}heorem for {G}abor-{T}oeplitz {L}ocalization {O}perators.
\newblock {\em Mich. Math. J.}, 49(1):13--21, 2001.

\bibitem{guille79}
Victor Guillemin.
\newblock Some {C}lassical {T}heorems in {S}pectral {T}heory {R}evisited.
\newblock In {\em Seminar on Singularities of Solutions of Linear Partial Differential Equations}, volume~91, pages 219--259. Princeton Univ. Press Princeton, NJ, 1979.

\bibitem{janssen1983szegHo}
AJEM Janssen and Steven Zelditch.
\newblock Szeg{\H{o}} {L}imit {T}heorems for the {H}armonic {O}scillator.
\newblock {\em T. Am. Math. Soc.}, 280(2):563--587, 1983.

\bibitem{SZ15}
Gabor Szeg{\"o}.
\newblock Ein {G}renzwertsatz {\"u}ber die {T}oeplitzschen {D}eterminanten einer reellen positiven {F}unktion.
\newblock {\em Math. Ann.}, 76(4):490--503, 1915.

\bibitem{Z90}
Kehe Zhu.
\newblock {\em Operator Theory in Function Spaces}.
\newblock Dekker, 1st edition, 1990.

\bibitem{Z12}
Kehe Zhu.
\newblock {\em Analysis on Fock spaces}, volume 263.
\newblock Springer Science \& Business Media, 2012.

\end{thebibliography}
\end{document}